\documentclass{amsart}
\usepackage{srcltx}
\usepackage{amsmath}
\usepackage{amsfonts}
\usepackage{amssymb}
\usepackage{amsthm}
\usepackage[T1]{fontenc}
\usepackage{hhline}                   
\usepackage{multirow}                 
\usepackage{array}                    
\usepackage{float}  
\usepackage{mathtools}
\usepackage{latexsym}
\usepackage{enumitem}
\usepackage[hidelinks]{hyperref}
\usepackage{amsrefs}

\def\C{\mathbb{C}}

\def\Q{\mathbb{Q}}
\def\N{\mathbb{N}}
\def\F{\mathbb{F}}
\def\P{\mathbb{P}}

\newcommand{\supp }{\mathrm{supp\,}}

\newtheorem{theo}{Theorem}[section]

\newtheorem{proposition}[theo]{Proposition}

\newtheorem{remark}[theo]{Remark}
\newtheorem{theorem}[theo]{Theorem}
\newtheorem{example}[theo]{Example}

\def\div{\mathsf{D}}
\def\pp{\mathsf{Princ}}
\def\FF{\mathbf{F}}

\def\bP{\mathsf{P}}

\begin{document}
	\title[On the coefficients of the L-polynomial]{On the coefficients of the Zeta-Function's L-polynomial for algebraic function fields over finite constant fields.}

	\author{M. Koutchoukali}
	\address{Aix Marseille Univ, CNRS, Centrale Marseille, I2M, Marseille, France. Institut de Math\'ematiques de Marseille, UMR 7373, CNRS, Aix-Marseille Universit\'e, case 930, F13288 Marseille cedex 9, France}
	\email{koutchoukali.mahdi@hotmail.fr}
	\email{mohamed-mahdi.koutchoukali@ac-aix-marseille.fr}
	\begin{abstract}
		We give an explicit formula of the coefficients of the Zeta-Function's L-polynomial for algebraic function fields over finite constant fields. 
		Thus, we deduce an expression of the class number of algebraic function fields defined over finite fields. Moreover, we give an application 
		of this formula in the case of the curves of defect 2 defined over $\F_2$.
		
	\end{abstract}
	
	\date{\today}
	
	\maketitle
	
	\section{Introduction}
	
	\subsection{Context and motivation}
	
	It is well known that the arithmetic and geometry of algebraic function fields is partly reflected in the properties of their Zeta-functions (cf. \cite{tsvl} of Tsfasman and Vladut).
	For example, it is the case if we are interested in the existence of non-special divisors of degree $g-1$ in algebraic function fields of genus 
	$g$ defined over a finite field.  Indeed, such an existence was proved in \cite{balb} or in \cite{bariro} from the properties of the Zeta-function's $L$-polynomial 
	for algebraic function fields over finite constant fields. In particular, the existence is guaranteed when the coefficients of the L-polynomial are all positive (which is the case for maximal function fields). 
	More generally, this is the case when the quantity $S=a_g+2\sum_{i=0}^{g-1}a_i$ is positive. So the main motivation behind this article is ultimately to show methods 
	for obtaining information about the sign of the coefficients or the sign of the  quantity $S$.	
	
	\subsection{New results and organisation}
	
	This paper is organized as follow. In section \ref{preli}, we introduce notations and we recall a few results which we will use to prove the main result Theorem \ref{AiFoncSi}. In section \ref{main}, we give a new explicit formula of the coefficients of the $L$-polynomial. So far, these coefficients were defined only by induction (see \cite{stic}). Moreover, we apply this formula to give a new expression of the class number $h$ (see notations in section \ref{preli}). Finally, in section \ref{exam} and for a defect 2 curves over $\F_2$, we use the explicit formula and a result on the reciprocals of the roots of $L$, given by J.P. Serre to establish a few properties of the coefficients concluding with their sign and variation. 
	
	\section{Preliminaries}\label{preli}
	
	\subsection{Notation.} 
	
	Let us recall the usual notation (for the basic notions related
	to an algebraic function field $\FF/\F_q$ see \cite{stic}). Let $\FF/\F_q$ be a          function field of genus $g$.
	For any integer $k\geq 1$ we denote by $\bP_k(\FF/\F_q)$ 
	the set of places of degree $k$, by $\mathcal{B}_k(\FF/\F_q)$ the 
	cardinality of this set and by $\bP(\FF/\F_q)=\cup_k \bP_k(\FF/\F_q)$.
	The divisor group of $\FF/\F_q$  is denoted by ${\div}(\FF/\F_q)$.
	If a divisor $D \in {\div}(\FF/\F_q)$ is such that
	$$D= \sum_{P \in \bP(\FF/\F_q)} n_P P,$$
	the support of $D$ is the following finite set
	$$\supp (D)=\{ P \in \bP(\FF/\F_q)~|~n_P \neq 0\}$$
	and its degree is
	$$\deg(D)=\sum_{P\in \bP(\FF/\F_q)} n_P\deg(P).$$
	We denote by ${\div}_n(\FF/\F_q)$ the set of divisors of degree $n$.
	We say that the divisor $D$ is \emph{effective} if for each $P\in \supp(D)$
	we have $n_P \geq 0$ and we denote ${\div}_n^+(\FF/\F_q)$ the set of effective     
	divisors of degree $n$ and $A_{n}=\# {\div}_n^+(\FF/\F_q)$.\\
	
	Let $x \in \FF/\F_q$, we denote by $(x)$ the divisor
	associated to the rational function $x$, namely
	$$(x)=\sum_{P\in\bP(\FF/\F_q)} v_P(x)P,$$
	where $v_P$ is the valuation at the place $P$.
	Such a divisor $(x)$ is called a principal divisor, and
	the set of  principal divisors is
	a subgroup of ${\div}_0(\FF/\F_q)$ denoted by ${\pp}(\FF/\F_q)$.
	The factor group 
	$${\mathcal C}(\FF/\F_q)={\div}(\FF_q)/{\pp}(\FF/\F_q)$$
	is called the divisor class group. If $D_1$ and $D_2$ are
	in the same class, namely if the divisor $D_1-D_2$ is principal,
	we will write $D_1 \sim D_2$. We will denote
	by $[D]$ the class of the divisor $D$.\\
	
	The \emph{dimension} of a divisor $D$, denoted by $\dim(D)$, is the dimension
	of the vector space ${\mathcal L}(D)$ defined by the formula
	$${\mathcal L}(D)= \{ x \in \FF ~|~(x)\geq -D\}\cup \{0\}.$$
	By Riemann-Roch Theorem we know that the dimension of this vector space is related to the genus of $\FF$ and the degree of $D$ by 
	\begin{equation} \label{RR}
		\dim (D) = \deg (D)-g+1 +\dim (\kappa-D),
	\end{equation}
	where $\kappa$ denotes a canonical divisor of $\FF/\F_q$ of degree $2g-2$.
	In this relation, the complementary term $i(D) =\dim (\kappa-D)$ is called the \emph{index of speciality} and is not easy to compute in general. In particular, a divisor $D$ is non-special when the index of speciality $i(D)$ is zero.\\
	
	If $D_1 \sim D_2$, the following holds
	$$\deg(D_1)=\deg(D_2), \quad \dim(D_1)=\dim(D_2),$$
	so that we can define the degree 
	$\deg([D])$ and the dimension $\dim([D])$ of a class.
	Since the degree of a principal divisor is $0$, we can define the subgroup
	${\mathcal C}(\FF/\F_q)^0$ of classes of degree $0$ divisors in ${\mathcal C}(\FF/\F_q)$.
	It is a finite group and we denote by $h$ its order, called the \emph{class number} of $\FF/\F_q$. Moreover if 
	$$L(t)=\sum_{i=0}^{2g} a_i t^i=\prod_{i=1}^{g} [(1-\alpha_i t)(1-\overline{\alpha_i} t)]$$
	with $|\alpha_i|=\sqrt{q}$ is the numerator of the Zeta function of $\FF/\F_q$, we have  
	$h=L(1)$. \\
	
	\subsection{Elementary results}
	
	The following results will be applied to prove the main results.\\
	We consider for $r \geq 1$ the number
	$$N_r := N(\FF_r) = |\{ P \in \P_{\FF_r}; deg(P)=1 \}|$$
	where $\FF_r = \FF \F_{q^r}$ is the constant field extension of $\FF/\F_q$ of degree       $r$. Let us remind the equation from \cite{stic}*{Corollary 5.1.16}, for all $r         \geq 1$,
	\begin{equation} 
		N_r = q^r+1-\sum_{i=1}^{2g} \alpha_i^r  \label{NiFoncAl} 
	\end{equation}
	where $\alpha_1, \ldots, \alpha_{2g} \in \C$ are the reciprocals of the roots of $L(t)$. In particular, since $N_1=N(\FF)$, we have
	$$ N(\FF) = q +1 - \sum_{i=1}^{2g} \alpha_i. $$
	
	\begin{proposition} \label{RelRec} \cite{stic}*{Corollary 5.1.17}
		Let $L(t) = \sum_{i=0}^{2g} a_i t^i$ be the $L$-polynomial of $\FF/\F_q$, and $S_r := N_r-(q^r+1)$. Then we have:
		\begin{enumerate}[label={(\alph*)}]
			\item $L'(t)/L(t) = \sum_{r=1}^{\infty} S_r t^{r-1}$.
			\item $a_0 = 1$ and 
			\begin{equation}
				i a_i = S_i a_0 + S_{i-1} a_1 + \cdots + S_1 a_{i-1} \label{RelRecEg}
			\end{equation}
			for $i=1, \ldots, g$.\\
		\end{enumerate}
		Given $N_1, \ldots, N_g$ and using $a_{2g-i}=q^{g-i}a_i$, we can determine $L(t)$ from \eqref{RelRecEg}.
	\end{proposition}
	
	Let 
	$$ \mathbf{S} = (S_1, S_2, S_3, \ldots) $$
	be an infinite sequence of integers with $S_i$ defined in Proposition \ref{RelRec}, then the equality \eqref{RelRecEg} can be defined as a linear infinite recurrence relation, for $n \geq 1$
	\begin{equation}
		a_n = \frac{1}{n} \sum_{i=1}^{\infty} S_i a_{n-i} \label{InfRec}
	\end{equation}
	with $a_0 = 1$ and $a_{<0}=0$.\\
	\\
	T. Goy and R. Zatorsky \cite{goza} studied a similar infinite recurrence using the parapermanent of a triangular matrix defined below. The same tools will be used in the proof of our main result. \\
	\\
	A triangular number table
	\begin{equation} \label{An}
		B_n =
		\begin{pmatrix}
			b_{1,1} &        &        & \\
			b_{2,1} & b_{2,2} &        & \\
			\vdots & \vdots & \ddots & \\
			b_{n,1} & b_{n,2} & \ldots & b_{n,n} 
		\end{pmatrix}
	\end{equation} 
	is a \textit{triangular matrix} of order $n$. We note $B_n= (b_{i,j})_{1 \leq j \leq i \leq n}$.\\
	To each element $b_{i,j}$ of matrix \eqref{An} we assign $(i-j+1)$ elements $b_{i,k}$, where $k=j, \ldots, i$, which are called the \textit{derived elements} of the matrix determined by the \textit{key element} $b_{i,j}$. The product of all derived elements determined by an element $b_{i,j}$ is called the \textit{factorial product} of the element $b_{i,j}$ and denoted by
	$$\{ b_{i,j} \} = \prod_{k=j}^{i} b_{i,k} .$$
	A tuple of key elements is called a \textit{normal tuple} if their derived elements form a set of elements of cardinality $n$, no two of which belong to the same column of the matrix.\\
	Let $\P(n)$ be the set of all ordered partitions of a positive integer $n$ into positive integer summands, we said that $\P(n)$ is the set of compositions of $n$, each integer of a composition is called \textit{a part}. We know, also, that $\#\P(n) = 2^{n-1}$. A one-to-one correspondence between elements of $\P(n)$ and normal tuples of key elements of matrix \eqref{An} was established by Tarakanov and Zatorsky \cite{taza}, namely,
	
	$$(n_1, n_2, \ldots, n_r) \in \P(n) \leftrightarrow\  \left\{\begin{array}{ll}
		\hbox{the normal tuple:}\ (b_{N_1,N_0+1},b_{N_2,N_1+1}, \ldots, b_{N_r,N_{r-1}+1}),\\
		\hbox{where}\ N_0 = 0\ \mbox{ and }\ N_s= \sum_{i=1}^{s} n_i\ \mbox{ for }\ s=1,2, \ldots,r.\\
	\end{array}
	\right.  $$
	
	The \textit{parapermanent} of the matrix \eqref{An} is defined as
	$$ pper(B_n) = \sum_{(m_1,\cdots,m_r) \in \P(n)} \prod_{s=1}^{r} \{ b_{i(s),j(s)} \} $$
	where $b_{i(s),j(s)}$ is the key element corresponding to $m_s$, the $s$-th component of the composition $(m_1, \ldots, m_r) \in \P(n)$. It can also be written as follow
	\begin{equation}
		pper(B_n) = \sum_{r=1}^{n}\ \sum_{m_1+\cdots+m_r=n}\ \prod_{s=1}^{r} \{ b_{m_1+\cdots+m_s , m_1+\cdots+m_{s-1}+1} \}. \label{Def2pper}
	\end{equation}
	
	and can be decomposed by elements of the last row as mentioned in \cite{zato}, \textit{i.e.},
	\begin{equation}
		pper(B_n) = \sum_{s=1}^{n} \{ b_{n,s} \} pper(B_{s-1}),  \label{Decpper}
	\end{equation}
	where by definition, $pper(B_0)=1$.
	
	\begin{example}
		Let us determine $pper(B_3)$, we have for\\
		
		$$r=3 \rightarrow\ (1,1,1) \rightarrow  \left\{\begin{array}{lll}
			s=1 \rightarrow \{b_{i(1),j(1)} \}=\{b_{1,1} \}= b_{1,1},\\
			s=2 \rightarrow \{b_{i(2),j(2)} \}=\{b_{2,2} \}= b_{2,2},\\
			s=3 \rightarrow \{b_{i(3),j(3)} \}=\{b_{3,3} \}= b_{3,3} \\
		\end{array}
		\right.  $$
		$$r=2 \rightarrow\ \left\{\begin{array}{ll} 
			(1,2) \rightarrow  \left\{\begin{array}{ll}
				s=1 \rightarrow \{b_{i(1),j(1)} \}=\{b_{1,1} \}= b_{1,1},\\
				s=2 \rightarrow \{b_{i(2),j(2)} \}=\{b_{3,2} \}= b_{3,2} \cdot                         b_{3,3},\\
				
			\end{array}
			\right.   \\
			(2,1) \rightarrow  \left\{\begin{array}{ll}
				s=1 \rightarrow \{b_{i(1),j(1)} \}=\{b_{2,1} \}= b_{2,1}                            \cdot b_{2,2},\\
				s=2 \rightarrow \{b_{i(2),j(2)} \}=\{b_{3,3} \}=                        b_{3,3}\\
				
			\end{array}
			\right.\\
		\end{array}
		\right.
		$$
		\\
		$$ r=1 \rightarrow (3) \rightarrow \{b_{i(1),j(1)} \}=\{b_{3,1} \}= b_{3,1}                            \cdot b_{3,2} \cdot b_{3,3} $$
		Thus,
		$$pper(B_3) = b_{1,1} \cdot b_{2,2} \cdot b_{3,3} + b_{1,1} \cdot b_{3,2} \cdot b_{3,3} + b_{2,1} \cdot b_{2,2} \cdot b_{3,3} +b_{3,1} \cdot b_{3,2} \cdot b_{3,3}.$$
	\end{example}
	
	\section{Main result} \label{main}
	
	In this section, we provide an explicit formula of the coefficients of the $L$-polynomial. From this formula, we deduce a new expression of the class number $h$. 
	\begin{theorem} \label{AiFoncSi}
		Let $L(t) = \sum_{i=0}^{2g} a_i t^i$ be the $L$-polynomial of $\FF/\F_q$, and \linebreak[4] $S_r := N_r-(q^r+1)$. For $i=1, \ldots, g$, the following equations are equivalent:
		\begin{enumerate}[label={(\alph*)}]
			\item $$a_i = \frac{1}{i} \sum_{j=1}^{\infty} S_j a_{i-j}$$
			where $a_0 = 1$ and $a_{<0}=0$.
			\item
			$$a_i = pper
			\begin{pmatrix}
				S_1 &        &        &     &\\
				\frac{S_2}{S_1} & \frac{S_1}{2} &        &  &\\
				\frac{S_3}{S_2}& \frac{S_2}{S_1} & \frac{S_1}{3} \\
				\vdots & \vdots & \vdots & \ddots &\\
				\frac{S_i}{S_{i-1}} & \frac{S_{i-1}}{S_{i-2}} & \frac{S_{i-2}}{S_{i-3}} & \ldots & \frac{S_1}{i}
			\end{pmatrix}
			$$
			\item $$ a_i = \sum_{(m_1,\ldots,m_r)\in \P(i)} \prod_{s=1}^{r} \frac{S_{m_s}}{m_1+\cdots+m_s} = \sum_{(m_1,\ldots,m_r)\in \P(i)} \prod_{s=1}^{r} \frac{N_{m_s}-(q^{m_s}+1)}{m_1+\cdots+m_s} $$
		\end{enumerate}
	\end{theorem}
	
	\begin{proof}
		Let
		$$B_n = 
		\begin{pmatrix}
			S_1 &        &        &     &\\
			\frac{S_2}{S_1} & \frac{S_1}{2} &        &  &\\
			\frac{S_3}{S_2}& \frac{S_2}{S_1} & \frac{S_1}{3} \\
			\vdots & \vdots & \vdots & \ddots &\\
			\frac{S_n}{S_{n-1}} & \frac{S_{n-1}}{S_{n-2}} & \frac{S_{n-2}}{S_{n-3}} & \ldots & \frac{S_1}{n}
		\end{pmatrix}
		$$
		First, we prove the equivalence between (a) and (b) by induction. By definition, $pper(B_0)= 1 =a_0$. Suppose that the equivalence is true for all $i < n$.
		by \eqref{Decpper}, we have 
		$$pper(B_n) = \sum_{j=1}^{n} \{ b_{n,j} \} pper(B_{j-1})$$
		by induction hypothesis, we have
		$$pper(B_{j-1}) = a_{j-1}$$
		finally, by construction we have
		$$ b_{n,j} = \frac{S_{n+1-j}}{S_{n-j}} \Leftrightarrow \{ b_{n,j} \} = \left\{ \frac{S_{n+1-j}}{S_{n-j}} \right\} = \frac{S_{n+1-j}}{n}.$$
		Thus, 
		$$a_n = pper(B_n) = \sum_{j=1}^{n} \left\{ \frac{S_{n+1-j}}{S_{n-j}} \right\} a_{j-1}$$
		$$a_n = \sum_{j=1}^{n}  \frac{S_{n+1-j}}{n} a_{j-1}$$
		$$ a_n = \frac{1}{n}\sum_{j=0}^{n-1}  S_{n-j} a_{j} $$
		
		Now, we prove the equivalence between (b) and (c). Since by \eqref{Def2pper}
		$$ pper(B_n) = \sum_{r=1}^{n} \sum_{m_1+\cdots+m_r=n} \prod_{s=1}^{r} \{ b_{m_1+\cdots+m_s , m_1+\cdots+m_{s-1}+1} \}$$
		$$pper(B_n) = \sum_{(m_1,\ldots,m_r) \in \P(n)} \prod_{s=1}^{r} \{ b_{m_1+\cdots+m_s , m_1+\dots+m_{s-1}+1} \}$$
		it is sufficient to notice, as in the first part of the proof, that
		$$ b_{n,j} = \frac{S_{n+1-j}}{S_{n-j}} \Leftrightarrow \{ b_{n,j} \} = \left\{ \frac{S_{n+1-j}}{S_{n-j}} \right\} = \frac{S_{n+1-j}}{n}$$
		
		thus
		$$\{ b_{m_1+\cdots+m_s,m_1+\cdots+m_{s-1}+1} \} = \frac{S_{m_1+\cdots+m_s+1-(m_1+\cdots+m_{s-1}+1)}}{m_1+\ldots+m_s}$$
		$$\{ b_{m_1+\cdots+m_s,m_1+\cdots+m_{s-1}+1} \} = \frac{S_{m_s}}{m_1+\cdots+m_s}.$$
	\end{proof}
	
	\begin{example}
		We want to determine $a_4$. We have
		
		$$\P(4) = \{(4),(1,3),(3,1),(2,2),(1,1,2),(1,2,1),(2,1,1),(1,1,1,1)\}$$
		
		\begin{equation*}
			r=1 \longrightarrow \{(4) \} \in \P(4) \longrightarrow \sum_{4} \prod_{s=1}^{1}\{ b_{m_1+\cdots+m_s,m_1+\cdots+m_{s-1}+1} \} = \{ b_{4,1} \} = \frac{S_4}{4}
		\end{equation*}
		
		\begin{equation*}
			\begin{split}
				r=2 \longrightarrow \{(1,3),(3,1),(2,2) \} \in \P(4) \longrightarrow \sum\limits_{\substack{1+3 \\ 3+1 \\ 2+2}} \prod_{s=1}^{2}\{ b_{m_1+\cdots+m_s,m_1+\cdots+m_{s-1}+1} \} = \\ \{ b_{1,1} \} \cdot \{ b_{4,2} \}+\{ b_{3,1} \} \cdot \{ b_{4,4} \}+ \{ b_{2,1} \} \cdot \{ b_{4,3} \} = \frac{S_1}{1}\frac{S_3}{4} + \frac{S_3}{3} \frac{S_1}{4}+ \frac{S_2}{2} \frac{S_2}{4}
			\end{split}
		\end{equation*}
		
		\begin{equation*}
			\begin{split}
				r=3 \longrightarrow \{(1,1,2),(1,2,1),(2,1,1) \} \in \P(4) \longrightarrow   \sum\limits_{\substack{1+1+2 \\ 1+2+1 \\ 2+1+1}} \prod_{s=1}^{3}\{ b_{m_1+\cdots+m_s,m_1+\cdots+m_{s-1}+1} \} = \\  \{ b_{1,1} \} \cdot \{ b_{2,2} \} \cdot \{ b_{4,3} \}+\{ b_{1,1} \} \cdot \{ b_{3,2} \} \cdot \{ b_{4,4} \}+ \{ b_{2,1} \} \cdot \{ b_{3,3} \} \cdot \{ b_{4,4} \} = \\ \frac{S_1}{1}\frac{S_1}{2} \frac{S_2}{4} + \frac{S_1}{1} \frac{S_2}{3} \frac{S_1}{4}+ \frac{S_2}{2} \frac{S_1}{3} \frac{S_1}{4}
			\end{split}
		\end{equation*}
		
		\begin{equation*}
			\begin{split}
				r=4 \longrightarrow \{(1,1,1,1) \} \in \P(4) \longrightarrow \sum_{1+1+1+1} \prod_{s=1}^{4}\{ b_{m_1+\cdots+m_s,m_1+\cdots+m_{s-1}+1} \} = \\ \{ b_{1,1} \} \cdot \{ b_{2,2} \} \cdot \{ b_{3,3} \} \cdot \{ b_{4,4} \} = \frac{S_1}{1}\frac{S_1}{2} \frac{S_1}{3} \frac{S_1}{4} 
			\end{split}
		\end{equation*}
		finally 
		\begin{eqnarray*}
			a_4 & = &\frac{S_4}{4}+\frac{S_1}{1}\frac{S_3}{4} + \frac{S_3}{3} \frac{S_1}{4}+ \frac{S_2}{2} \frac{S_2}{4}+\frac{S_1}{1}\frac{S_1}{2} \frac{S_2}{4} + \frac{S_1}{1} \frac{S_2}{3} \frac{S_1}{4}+ \frac{S_2}{2} \frac{S_1}{3} \frac{S_1}{4}+\frac{S_1}{1}\frac{S_1}{2} \frac{S_1}{3} \frac{S_1}{4} \\
			& = &\frac{S_4}{4}+\frac{S_1S_3}{3}+\frac{S_2^2}{8}+\frac{S_1^2S_2}{4}+\frac{S_1^4}{24}.
		\end{eqnarray*}
		with this, we conclude our example.
	\end{example}
	
	We can also give an explicit formula of the coefficients $a_i$ depending on the reciprocals of the roots of $L$. But first, let's give an expression of $N_m$ depending on these roots, see \cite{stic}*{Corollary 5.1.16}:
	
	$$N_m =  q^m+1- \sum_{i=1}^{2g} \alpha_i^m $$

	which can also be written 
	
	\begin{equation}
		N_m = q^m+1- \sum_{i=1}^{g} 2q^{m/2} cos(\pi m \phi_i)  \label{Mobius}
	\end{equation}
	
	where $(\alpha_1,\ldots,\alpha_{2g}) = (q^{1/2}e^{i\pi \phi_1},\ldots,q^{1/2}e^{i\pi \phi_g},q^{1/2}e^{-i\pi \phi_1},\ldots,q^{1/2}e^{-i\pi \phi_g})$ are the reciprocals of the roots of $L(t)$ with $\phi_i \in [0,1]$.
	The following results are the consequences of Theorem \ref{AiFoncSi}, \eqref{NiFoncAl} and \eqref{Mobius}.
	
	\begin{proposition} \label{CoefRr}
		Let $L(t) = \sum_{i=0}^{2g} a_i t^i$ be the $L$-polynomial of $\FF/\F_q$. We have
		\begin{eqnarray*}
			a_i & = &\sum_{(m_1,\ldots,m_r)\in \P(i)} \prod_{s=1}^{r} (-1)^r \frac{\sum_{j=1}^{g}2q^{m_s/2}cos(\pi m_s \phi_j)}{m_1+\cdots+m_s}  \\
			& = &\sum_{(m_1,\ldots,m_r)\in \P(i)} \prod_{s=1}^{r} (-1)^r \frac{\sum_{j=1}^{2g}\alpha_j^{m_s}}{m_1+\cdots+m_s} 
		\end{eqnarray*} 
		for $i=1, \ldots, g$ and $a_{2g-i}=q^{g-i}a_i$ with $a_0=1$.
	\end{proposition}
	
	\begin{proof}
		We use Theorem \ref{AiFoncSi} (3) and
		$$N_m = q^m+1- \sum_{i=1}^{g} 2q^{m/2} cos(\pi m \phi_i)= q^m+1- \sum_{i=1}^{2g} \alpha_i^m.$$
	\end{proof}
	
	We conclude by a new expression of the class number $h$ depending on the previous parameters.
	
	\begin{proposition} \label{h}
		Let $L(t) = \sum_{i=0}^{2g} a_i t^i$ be the $L$-polynomial of $\FF/\F_q$ and $h$ the class number. We have
		
		\begin{eqnarray*}
			h& = &1+q^g+\sum_{i=1}^{g-1}(1+q^{g-i}) \sum_{(m_1,\ldots,m_r)\in \P(i)} \prod_{s=1}^{r} (-1)^r \frac{\sum_{j=1}^{g}2q^{m_s/2}cos(\pi m_s \phi_j)}{m_1+\cdots+m_s}+ \\
			&   & \sum_{(m_1,\ldots,m_r)\in \P(g)} \prod_{s=1}^{r} (-1)^r \frac{\sum_{j=1}^{g}2q^{m_s/2}cos(\pi m_s \phi_j)}{m_1+\cdots+m_s}\\
			& = &1+q^g+\sum_{i=1}^{g-1}(1+q^{g-i}) \sum_{(m_1,\ldots,m_r)\in \P(i)} \prod_{s=1}^{r} (-1)^r \frac{\sum_{j=1}^{2g}\alpha_j^{m_s}}{m_1+\cdots+m_s} + \\
			&   &\sum_{(m_1,\ldots,m_r)\in \P(g)} \prod_{s=1}^{r} (-1)^r \frac{\sum_{j=1}^{2g}\alpha_j^{m_s}}{m_1+\cdots+m_s}\\
			& = & 1+q^g+\sum_{i=1}^{g-1}(1+q^{g-i}) \sum_{(m_1,\ldots,m_r)\in \P(i)} \prod_{s=1}^{r} \frac{N_{m_s}-(q^{m_s}+1)}{m_1+\cdots+m_s}+ \\
			&   &\sum_{(m_1,\ldots,m_r)\in \P(g)} \prod_{s=1}^{r} \frac{N_{m_s}-(q^{m_s}+1)}{m_1+\cdots+m_s}
		\end{eqnarray*}
	\end{proposition}
	
	\begin{proof}
		Since $h=L(1)=\sum_{i=0}^{2g}a_i$ with $a_0=1$ and $a_{2g}=q^g a_0$, this result is a consequence of Proposition \ref{CoefRr}. Note that the case $i=g$ has to be treated separately, otherwise $a_g$ will be counted twice.  
	\end{proof}

	\section{Example of application} \label{exam}
	
	In this section, we will exhibit an application for the explicit formula in Proposition \ref{CoefRr}. This application is not the simple calculation of the coefficients of the $L$-polynomial, since one can easily show that the complexity of the formula in \ref{AiFoncSi} is exponential.\\
	This application consists of giving some of these coefficients properties. These latter will allow us to determine the sign and the variation of the sequence $(a_n)_{ 0 \leq n \leq g}$.\\    
	We will focus on curves of defect 2 over the finite field $\F_2$ \textit{i.e.}
	$$ |N_1(\FF/\F_2)-(q+1)| = |N_1(\FF/\F_2)-3| = g[2\sqrt{q}]-2 = g[2\sqrt{2}]-2 = 2g-2,$$
	where $[x]$ denotes the largest integer $\leq x$. The general case was studied By J.P. Serre in \cite{serre}*{Theorem 2.5.1}.
	
	\begin{theorem} \label{SCase}
		For a curve over $\F_q$ such that $|N_1(\FF/\F_q)-(q+1)| = g[2\sqrt{q}]-2$, there are six possibilities for $(\alpha_1+\overline{\alpha_1}, \ldots, \alpha_g+\overline{\alpha_g})$ : 
		\begin{enumerate}[label={(\alph*)}]
			\item $\pm([2\sqrt{q}], \ldots,[2\sqrt{q}],[2\sqrt{q}]-2)$ for $g\geq 1$.
			\item $\pm([2\sqrt{q}],\ldots,[2\sqrt{q}],[2\sqrt{q}]+\sqrt{2}-1,[2\sqrt{q}]-\sqrt{2}-1)$ for $g=2$.
			\item $\pm([2\sqrt{q}],\ldots,[2\sqrt{q}],[2\sqrt{q}]-1,[2\sqrt{q}]-1)$ for $g=2$.
			\item $\pm([2\sqrt{q}],\ldots,[2\sqrt{q}],[2\sqrt{q}]+\sqrt{3}-1,[2\sqrt{q}]-\sqrt{3}-1)$ for $g \geq 2$.
			\item $\pm([2\sqrt{q}],\ldots,[2\sqrt{q}],[2\sqrt{q}]+\frac{-1+\sqrt{5}}{2},[2\sqrt{q}]+\frac{-1+\sqrt{5}}{2},[2\sqrt{q}]+\frac{-1-\sqrt{5}}{2},[2\sqrt{q}]+\frac{-1-\sqrt{5}}{2})$ for $g=4$.
			\item $\pm([2\sqrt{q}],\ldots,[2\sqrt{q}],[2\sqrt{q}]+1-4cos^2(\frac{\pi}{7}),[2\sqrt{q}]+1-4cos^2(\frac{2\pi}{7}),[2\sqrt{q}]+1-4cos^2(\frac{3\pi}{7}))$ for $g=3$.
			
		\end{enumerate}
	\end{theorem}

	We are interested here by the case (a) of Theorem \ref{SCase} for $q=2$. This case offers possibly infinitely many examples of curves.
	
	Let $q=2$, we can easily deduce that under the conditions of Theorem \ref{SCase} case (a), we have $(\alpha_1+\overline{\alpha_1}, \ldots, \alpha_g+\overline{\alpha_g})=\pm([2\sqrt{q}], \ldots,[2\sqrt{q}],[2\sqrt{q}]-2)=\pm(2,\ldots,2,0)$. Thus, with equation \eqref{Mobius} notations we have $\phi_i = 1/4$ or $3/4$ since $2\sqrt{2}cos(\pi \cdot \phi_i)=\pm 2$ for $i \in \{1, \ldots, g-1 \}$ and $\phi_g=1/2$ since $2\sqrt{2}cos(\pi \cdot \phi_g)=0$.
	
	\begin{proposition} \label{ForC1}
		Let $L(t) = \sum_{n=0}^{2g} a_n t^n$ be the $L$-polynomial of $\FF/\F_2$, under the conditions of case (a) we have
		$$ a_n= a_{n,\pi/4} = \sum_{(m_1,\ldots,m_r)\in \P(n)} (-1)^r \cdot 2^r \cdot 2^{n/2} \prod_{s=1}^{r}  \frac{(g-1)cos(m_s \cdot \pi/4)+cos(m_s \cdot \pi/2)}{m_1+\cdots+m_s}$$
		or
		$$ a_n= a_{n,3\pi/4} = \sum_{(m_1,\ldots,m_r)\in \P(n)} (-1)^r \cdot 2^r \cdot 2^{n/2} \prod_{s=1}^{r}  \frac{(g-1)cos(m_s \cdot 3 \pi/4)+cos(m_s \cdot \pi/2)}{m_1+\cdots+m_s}$$
		for $n=1, \ldots, g$ and $a_{2g-n}=q^{g-n}a_n$ with $a_0=1$.
	\end{proposition}
	
	\begin{proof}
		By Proposition \ref{CoefRr} and the conditions of case (a), we have
		\begin{eqnarray*}
			a_{n,\pi/4} & = &\sum_{(m_1,\ldots,m_r)\in \P(n)} \prod_{s=1}^{r} (-1)^r 
			\frac{\sum_{j=1}^{g}2q^{m_s/2}cos(\pi m_s \phi_j)}{m_1+\cdots+m_s}  \\
			& = & \sum_{(m_1,\ldots,m_r)\in \P(n)} \prod_{s=1}^{r} (-1)^r \frac{(\sum_{j=1}^{g-1}2q^{m_s/2}cos(m_s \cdot \pi/4))+2q^{m_s/2}cos(m_s \cdot \pi/2)}{m_1+\cdots+m_s}  \\
			& = &\sum_{(m_1,\ldots,m_r)\in \P(n)} (-1)^r \cdot 2^r \cdot 2^{n/2} \prod_{s=1}^{r} \frac{(g-1)cos(m_s \cdot \pi/4)+cos(m_s \cdot \pi/2)}{m_1+\cdots+m_s} \\
		\end{eqnarray*}
		The second equality can be proven similarly. 
	\end{proof}
	
	For what comes next, we need the following notations:\\
	$E_n:=\{n+8k, k \in \N^*\}$.\\
	$C_{\pi/4}(n) := (g-1)cos(n \cdot \pi/4)+cos(n \cdot \pi/2)$. \\
	$C_{3\pi/4}(n) := (g-1)cos(n \cdot 3\pi/4)+cos(n \cdot \pi/2)$.\\
	$\P_{\pi/4}^+(n) := \{(m_1, \ldots, m_r) \in \P(n) | (-1)^r \cdot 2^r \cdot 2^{n/2} \prod_{s=1}^{r} \frac{C_{\pi/4}(m_s)}{m_1+\cdots+m_s} > 0 \}$. \\
	$\P_{\pi/4}^-(n) := \{(m_1, \ldots, m_r) \in \P(n) | (-1)^r. 2^r. 2^{n/2} \prod_{s=1}^{r} \frac{C_{\pi/4}(m_s)}{m_1+\cdots+m_s} < 0 \}$. \\
	$\P_{3\pi/4}^+(n) := \{(m_1, \ldots, m_r) \in \P(n) | (-1)^r \cdot 2^r \cdot 2^{n/2} \prod_{s=1}^{r} \frac{C_{3\pi/4}(m_s)}{m_1+\cdots+m_s} > 0 \}$. \\
	$\P_{3\pi/4}^-(n) := \{(m_1, \ldots, m_r) \in \P(n) | (-1)^r \cdot 2^r \cdot 2^{n/2} \prod_{s=1}^{r} \frac{C_{3\pi/4}(m_s)}{m_1+\cdots+m_s} < 0 \}$. \\
	$P_{\pi/4}^+(n) := \# \P_{\pi/4}^+(n)$, $P_{\pi/4}^-(n) = \# \P_{\pi/4}^-(n)$, $P_{3\pi/4}^+(n) = \# \P_{3\pi/4}^+(n)$, \linebreak[4] $P_{3\pi/4}^-(n) = \# \P_{3\pi/4}^-(n)$. \\
	$CR_{\pi/4}: \P(n) \rightarrow \Q$ the map defined by: $$CR_{\pi/4}((m_1,\ldots, m_r)) = (-1)^r \cdot 2^r \cdot 2^{n/2} \prod_{s=1}^{r} \frac{C_{\pi/4}(m_s)}{m_1+\cdots+m_s}. $$ 
	$CR_{3\pi/4}: \P(n) \rightarrow \Q$ the map defined by: $$CR_{3\pi/4}((m_1,\ldots, m_r)) = (-1)^r \cdot 2^r \cdot 2^{n/2} \prod_{s=1}^{r} \frac{C_{3\pi/4}(m_s)}{m_1+\cdots+m_s}. $$ 
	
	\begin{proposition} \label{RelRec2}
		Let $L(t) = \sum_{n=0}^{2g} a_n t^n$ be the $L$-polynomial of $\FF/\F_2$.  Under the conditions of case (a) we have
		$$a_n= \sum\limits_{\substack{(m_1,\ldots,m_r)\in \P(n) \\ m_r \neq i }} (-1)^r \cdot 2^r \cdot 2^{n/2} \prod_{s=1}^{r}  \frac{C_{\theta}(m_s)}{m_1+\cdots+m_s} - 2^{\frac{i+2}{2}} \frac{C_{\theta}(i)}{n} a_{n-i}$$
		and
		$$a_n = \sum_{i=1}^{n} -2^{\frac{i+2}{2}} \cdot \frac{C_{\theta}(i)}{n} a_{n-i} $$
		for $n=1, \ldots, g$ and $a_{2g-n}=q^{g-n}a_n$ with $a_0=1$ and $\theta = \pi/4$ or $3\pi/4$.
		
	\end{proposition}
	
	\begin{proof}
		The main idea is to notice that 
		$$\P(n) = \bigcup_{i=1}^{n} \{(m_1,\ldots,m_r,i)\ \hbox{with}\ (m_1,\ldots,m_r) \in \P(n-i) \}$$
		then 
		
		\begin{eqnarray*}
			a_n & = & \sum_{(m_1,\ldots,m_r)\in \P(n)} (-1)^r \cdot 2^r \cdot 2^{n/2} \prod_{s=1}^{r}                      \frac{C_{\theta}(m_s)}{m_1+\cdots+m_s} \\
			& = & \sum\limits_{\substack{(m_1,\ldots,m_r)\in \P(n) \\ m_r \neq i }} (-1)^r \cdot 2^r \cdot 
			2^{n/2} \prod_{s=1}^{r}  \frac{C_{\theta}(m_s)}{m_1+\cdots+m_s} +  \\
			&   &\sum_{(m_1,\ldots,m_{r-1},i)\in \P(n)} (-1)^r \cdot 2^r \cdot 2^{n/2} \prod_{s=1}^{r}             \frac{C_{\theta}(m_s)}{m_1+\cdots+m_s}\\
			& = & \sum\limits_{\substack{(m_1,\ldots,m_r)\in \P(n) \\ m_r \neq i }} (-1)^r \cdot 2^r \cdot 
			2^{n/2} \prod_{s=1}^{r}  \frac{C_{\theta}(m_s)}{m_1+\cdots+m_s} + \\
			&   & -2 \cdot 2^{i/2}\frac{C_{\theta}(i)}{n} \cdot \sum_{(m_1,\ldots,m_{r-1},i)\in \P(n)} 
			(-1)^{r-1} \cdot 2^{r-1} \cdot 2^{(n-i)/2}\prod_{s=1}^{r-1}\frac{C_{\theta}(m_s)}{m_1+\cdots+m_s} \\ 
			& = & \sum\limits_{\substack{(m_1,\ldots,m_r)\in \P(n) \\ m_r \neq i }} (-1)^r \cdot 2^r \cdot 
			2^{n/2} \prod_{s=1}^{r}  \frac{C_{\theta}(m_s)}{m_1+\cdots+m_s} +\\
			&   & -2 \cdot 2^{i/2}\frac{C_{\theta}(i)}{n} \cdot \sum_{(m_1,\ldots,m_{r-1})\in \P(n-i)} 
			(-1)^{r-1} \cdot 2^{r-1} \cdot 2^{(n-i)/2}\prod_{s=1}^{r-1}\frac{C_{\theta}(m_s)}{m_1+\cdots+m_s} \\
			& = & \sum\limits_{\substack{(m_1,\ldots,m_r)\in \P(n) \\ m_r \neq i }} (-1)^r \cdot 2^r \cdot 
			2^{n/2} \prod_{s=1}^{r}  \frac{C_{\theta}(m_s)}{m_1+\cdots+m_s} - 2^{(i+2)/2} \frac{C_{\theta}(i)}{n} a_{n-i} 
		\end{eqnarray*}
		
		The second equality can be proven similarly. 
		
	\end{proof}
	
	\begin{remark}
		The second equality could also be proved applying Proposition \ref{RelRec}.   
	\end{remark}
	
	Notice that 
	$$ C_{\pi/4}(m_s) = \left\{\begin{array}{lllll}
		(g-1) \sqrt{2}/2 & \mbox{if}\ m_s \in E_1 \cup E_7 \\
		-(g-1) \sqrt{2}/2 & \mbox{if}\ m_s \in E_3 \cup E_5 \\
		-1 & \mbox{if}\ m_s \in E_2 \cup E_6 \\
		-(g-2) & \mbox{if}\ m_s \in E_4 \\
		g & \mbox{if}\ m_s \in E_8 \\
	\end{array}
	\right.  $$
	and
	$$ C_{3\pi/4}(m_s) = \left\{\begin{array}{lllll}
		(g-1) \sqrt{2}/2 & \mbox{if}\ m_s \in E_3 \cup E_5 \\
		-(g-1) \sqrt{2}/2 & \mbox{if}\ m_s \in E_1 \cup E_7 \\
		-1 & \mbox{if}\ m_s \in E_2 \cup E_6 \\
		-(g-2) & \mbox{if}\ m_s \in E_4 \\
		g & \mbox{if}\ m_s \in E_8 \\
	\end{array}
	\right.  $$
	it is clear (for $g > 2$) that $m \in \P_{\pi/4}^+(n) $, if and only if, the number of parts of $m$ that are in $E_2 \cup E_3 \cup E_4 \cup E_5 \cup E_6$ has the same parity as $r$ the number of all parts. $m \in \P_{\pi/4}^-(n) $, if and only if, this number and $r$ have not the same parity. \\
	Similarly, we can characterize $\P_{3\pi/4}^+(n) $ and $\P_{3\pi/4}^-(n) $ by replacing $E_3$ and $E_5$ with $E_1$ and $E_7$.\\
	
	The following proposition, is a second characterization which is less obvious.   
	
	\begin{proposition} \label{Car}
		Let $L(t) = \sum_{n=0}^{2g} a_n t^n$ be the $L$-polynomial of $\FF/\F_2$, \linebreak[4] $m=(m_1,\ldots, m_r) \in \P(n)$ a composition of an integer $n \leq g$ with $g > 2$. Under the conditions of case (a) we have
		\begin{enumerate}
			\item[\rm{i})] The cardinality of $(E_1 \cup E_7 \cup E_8) \cap \{m_1,\ldots,m_r \}$ is even, if and only if, $m \in \P_{\pi/4}^+(n) $. It is odd, if and only if, $m \in \P_{\pi/4}^-(n)$.
			\item[\rm{ii})] The cardinality of $(E_3 \cup E_5 \cup E_8) \cap \{m_1,\ldots,m_r \}$ is even, if and only if, $m \in \P_{3\pi/4}^+(n) $. It is odd, if and only if, $m \in \P_{3\pi/4}^-(n)$.
		\end{enumerate}
	\end{proposition}
	
	\begin{proof}
		\begin{enumerate}
			\item[\rm{i})] By induction, we have for $n=1$: $\P(1) = \{(1) \}$ then $CR_{\pi/4}((1)) = -1 \cdot 2 \cdot 2^{1/2} \cdot \frac{(g-1)\sqrt{2}/2}{1}=-2(g-1) < 0$. \\
			For $n=2$: $\P(2) = \{(1,1),(2) \}$ then 
			$$CR_{\pi/4}((1,1)) = (-1)^2 \cdot 2^2 \cdot 2^{2/2} \cdot \frac{(g-1)\sqrt{2}/2}{1} \cdot \frac{(g-1)\sqrt{2}/2}{2}=2(g-1)^2 > 0,$$
			and
			$$CR_{\pi/4}((2)) = -1 \cdot 2 \cdot 2^{2/2} \cdot \frac{-1}{2}=2 > 0.$$
			Suppose that the equivalence is true for $n' < n$. Let \linebreak[4] $m=(m_1,\ldots,m_r) \in \P_{\pi/4}^+(n)$, then:
			\begin{itemize}
				\item If $(E_1 \cup E_7 \cup E_8) \cap \{m_1,\ldots,m_r \} = \emptyset$, the cardinality of this set is even.
				\item If $m_i < n$ such that $\{m_i \} \subset (E_1 \cup E_7 \cup E_8) \cap \{m_1,\ldots,m_r \} $, we have
				\begin{eqnarray*}
					CR_{\pi/4}(m) & = & (-1)^r \cdot 2^r \cdot 2^{n/2} \prod_{s=1}^{r} \frac{C_{\pi/4}(m_s)}          
					{m_1+\cdots+m_s} \\
					& = & (-1)^r \cdot 2^r \cdot 2^{n/2} \prod_{s=1}^{i-1} \frac{C_{\pi/4}(m_s)}          
					{m_1+\cdots+m_s} \cdot \frac{C_{\pi/4}(m_i)}{m_1+\cdots+m_i}. \\
					&   &\prod_{s=i+1}^{r} \frac{C_{\pi/4}(m_s)}{m_1+\cdots+m_i+\cdots+m_s}\\
					& = & (-1)^r \cdot 2^r \cdot 2^{n/2} \prod_{s=1}^{i-1} \frac{C_{\pi/4}(m_s)}          
					{m_1+\cdots+m_s} \cdot \frac{C_{\pi/4}(m_i)}{m_1+\cdots+m_i}. \\
					&   & \prod_{s=i+1}^{r} \frac{C_{\pi/4}(m_s)}{m_1+\cdots+m_{i- 
							1}+m_{i+1}+\cdots+m_s}. \\ 
					&   &\prod_{s=i+1}^{r} \frac{m_1+\cdots+m_{i-1}+m_{i+1}+\cdots+m_s} 
					{m_1+\cdots+m_{i-1}+m_{i}+m_{i+1}+\cdots+m_s} \\
					& = & -2.2^{m_i/2} C_{\pi/4}(m_i).\prod_{s=i+1}^{r} \frac{m_1+\cdots+m_{i- 
							1}+m_{i+1}+\cdots+m_s}{m_1+\cdots+m_{i-1}+m_{i}+m_{i+1}+\cdots+m_s}. \\
					&   &(-1)^{r-1} \cdot 2^{r-1} \cdot 2^{(n-m_i)/2} \prod_{s=1}^{r-1} \frac{C_{\pi/4} 
						(m_s)}{m_1+\cdots+m_s} \\
					& = & -2 \cdot 2^{m_i/2} C_{\pi/4}(m_i).\prod_{s=i+1}^{r} \frac{m_1+\cdots+m_{i- 
							1}+m_{i+1}+\cdots+m_s}{m_1+\cdots+m_{i}+\cdots+m_s}.\\
					&   & CR_{\pi/4}((m_1,\cdots,m_{i-1},m_{i+1},\cdots,m_r))     
				\end{eqnarray*} 
				Since $C_{\pi/4}(m_i) > 0$ and $CR_{\pi/4}(m) > 0$, we have $$CR_{\pi/4}((m_1,\ldots,m_{i-1},m_{i+1},\ldots,m_r)) < 0.$$ 
				By induction hypothesis, the cardinality of $$(E_1 \cup E_7 \cup E_8) \cap \{ m_1,\ldots,m_{i-1},m_{i+1},\ldots,m_r\}$$ is odd, thus the cardinality of $(E_1 \cup E_7 \cup E_8) \cap \{m_1,\ldots,m_r \}$ is even. 
				
			\end{itemize}
			
			\item[\rm{ii})] The second equivalence can be proven similarly by replacing $E_3$, $E_5$ with $E_1$, $E_7$.
			
		\end{enumerate}
	\end{proof}

	\begin{proposition} \label{Sym}
		Let $L(t) = \sum_{n=0}^{2g} a_n t^n$ be the $L$-polynomial of $\FF/\F_2$. Under the conditions of case (a) and with Proposition \ref{ForC1} notations, we have 
		$$ a_{n,\pi/4} = \left\{\begin{array}{ll}
			a_{n,3\pi/4}\ \hbox{if}\ n\ \hbox{is even} \\
			- a_{n,3\pi/4}\ \hbox{if}\ n\ \hbox{is odd}\\
		\end{array}
		\right.  $$
		for $n=1, \ldots, g$ and $a_{2g-n}=q^{g-n}a_n$ with $a_0=1$.
	\end{proposition}

	\begin{proof}
		Let $m=(m_1,\ldots, m_r) \in \P(n)$ a composition of an integer $n \leq g$. The main idea here, is to notice that if $n$ is even then the cardinality of
		$$(E_3 \cup E_5 \cup E_8) \cap \{m_1,\ldots,m_r \}$$
		and 
		$$(E_1 \cup E_7 \cup E_8) \cap \{m_1,\ldots,m_r \}$$
		are both even or both odd (1, 3, 5, 7 are the only odd numbers modulo 8). If $n$ is odd, then one is even and the second is odd. Thus, by Proposition \ref{Car} 
		
		$$ CR_{\pi/4}(m) = \left\{\begin{array}{ll}
			CR_{3\pi/4}(m)\ \hbox{if}\ n\ \hbox{is even} \\
			- CR_{3\pi/4}(m)\ \hbox{if}\ n\ \hbox{is odd}\\
		\end{array}
		\right.  $$
		we conclude by noticing that
		$$a_{n,\pi/4} = \sum_{m \in \P(n)} CR_{\pi/4}(m)\ \hbox{and}\ a_{n,3\pi/4} = \sum_{m \in \P(n)} CR_{3\pi/4}(m)$$
		
	\end{proof}

	\begin{proposition} \label{Ordre}
		Let $n \leq g$ be an integer with $g > 2$. then $P_{\pi/4}^+(n) > P_{\pi/4}^-(n)$ if $n$ is even, $P_{\pi/4}^+(n) < P_{\pi/4}^-(n)$ if $n$ is odd and $P_{3\pi/4}^+(n) > P_{3\pi/4}^-(n)$ for all $n$. More precisely
		
		\begin{enumerate}[label={(\alph*)}]
			\item $P_{\pi/4}^+(2) - P_{\pi/4}^-(2) = 2$.
			\item $P_{\pi/4}^-(3) - P_{\pi/4}^+(3) = 2$.
			\item $P_{\pi/4}^+(4) - P_{\pi/4}^-(4) = 4$.
			\item $P_{\pi/4}^-(5) - P_{\pi/4}^+(5) = 4$.
			\item $P_{\pi/4}^+(n) - P_{\pi/4}^-(n) > n$ for $n \geq 6$ if $n$ is even.
			\item $P_{\pi/4}^-(n) - P_{\pi/4}^+(n) > n$ for $n \geq 6$ if $n$ is odd.
			\item $P_{\pi/4}^\delta(n):=|P_{\pi/4}^+(n) - P_{\pi/4}^-(n)|$ is a strictly increasing sequence for $n \geq 6$.
			\item $P_{3\pi/4}^+(n) - P_{3\pi/4}^-(n) = 2$ for $n = 2$ and $3$.
			\item $P_{3\pi/4}^+(n) - P_{3\pi/4}^-(n) = 4$ for $n = 4$ and $5$.
			\item $P_{3\pi/4}^+(n) - P_{3\pi/4}^-(n) > n$ and $P_{3\pi/4}^\delta(n):=P_{3\pi/4}^+(n) - P_{3\pi/4}^-(n)$ is a strictly increasing sequence for $n \geq 6$.
		\end{enumerate}
		
	\end{proposition}
	
	\begin{proof}
		The equalities (a), (b), (c), (d), (h) and (i) follow from a simple calculation. (e), (f) and (g) follow from (j) by noticing that
		$$ CR_{\pi/4}(m) = \left\{\begin{array}{ll}
			CR_{3\pi/4}(m)\ \hbox{if}\ n\ \hbox{is even} \\
			- CR_{3\pi/4}(m)\ \hbox{if}\ n\ \hbox{is odd}\\
		\end{array}
		\right.  $$
		see the proof of Proposition \ref{Sym}.\\
		Let's prove inequality (j). For $n=6$, $P_{3\pi/4}^+(6) - P_{3\pi/4}^-(6) = 8 > 6$. For $n=7$, $P_{3\pi/4}^+(7) - P_{3\pi/4}^-(7) = 10 > 7$. By induction, suppose that for $n' < n$ the property (j) is true. By Proposition \ref{RelRec2}
		$$a_{n,3\pi/4} = \sum_{i=1}^{n} -2^{(i+2)/2} \cdot \frac{C_{3\pi/4}(i)}{n} a_{n-i}$$
		thus,\\
		for $1 \leq i \leq n$, $P_{3\pi/4}^+(n-i) - P_{3\pi/4}^-(n-i) > n-i$ and $P_{3\pi/4}^\delta(n-i)$ is strictly increasing by induction hypothesis,\\
		for $i \in E_1\cup E_2 \cup E_4 \cup E_6 \cup E_7$, $-2^{(i+2)/2} \cdot \frac{C_{3\pi/4}(i)}{n} > 0$. The difference between the positive and the negative terms of $-2^{(i+2)/2} \cdot \frac{C_{3\pi/4}(i)}{n} a_{n-i}$ is equal to $P_{3\pi/4}^\delta(n-i)$. \\
		and for $i \in E_3\cup E_5 \cup E_8$, $-2^{(i+2)/2} \cdot \frac{C_{3\pi/4}(i)}{n} < 0$. the positive and the negative terms of $-2^{(i+2)/2}.\frac{C_{3\pi/4}(i)}{n} a_{n-i}$ are reversed, the difference is equal to $-P_{3\pi/4}^\delta(n-i)$. \\ 
		We conclude that 
		$$P_{3\pi/4}^\delta(n)= \sum_{i \in E_1\cup E_2 \cup E_4 \cup E_6 \cup E_7} P_{3\pi/4}^\delta(n-i)- \sum_{i \in E_3\cup E_5 \cup E_8}P_{3\pi/4}^\delta(n-i). $$
		Finally, if $i \in E_3\cup E_5 \cup E_8$, we have $P_{3\pi/4}^\delta(n-i+1)-P_{3\pi/4}^\delta(n-i) \geq 1$ for $n-i \neq 2$ or $4$, and $P_{3\pi/4}^\delta(n-i+1)-P_{3\pi/4}^\delta(n-i) = 0$ for $n-i=2$ or $4$. Since $P_{3\pi/4}^\delta(n-1) > n-1$, we conclude that $P_{3\pi/4}^\delta(n) > n$ and $P_{3\pi/4}^\delta(n) > P_{3\pi/4}^\delta(n-1)$.
	\end{proof}

	\begin{theorem} \label{SignAi}
		Let $L(t) = \sum_{n=0}^{2g} a_n t^n$ be the $L$-polynomial of a function field $\FF/\F_2$ with $1 \leq g \leq 6$. Under the conditions of case (a) and for $n \in \{0, \ldots, g \}$ we have
		\begin{enumerate}[label={(\alph*)}]
			\item $a_{n,\pi/4} < 0 $ if $n$ is odd and $a_{n,\pi/4} > 0$ if $n$ is even.
			\item $a_{n,3\pi/4} > 0 $.
			\item The sequence $|a_{n,\theta}|$ with $\theta = \pi/4$ or $3\pi/4$ is increasing. 
		\end{enumerate}
	\end{theorem}

	\begin{proof}
		We will prove the theorem for $\theta = \pi/4$, the second part is obvious by Proposition \ref{Sym}.\\
		By Proposition \ref{ForC1} we have
		
		$$a_{n,\pi/4}  =  \sum_{(m_1,\ldots,m_r)\in \P(n)} (-1)^r. 2^{r+n/2} \prod_{s=1}^{r} \frac{(g-1)cos(m_s \cdot \pi/4)+cos(m_s \cdot \pi/2)}{m_1+\cdots+m_s} $$
		
		Suppose that $n$ is even. Let $m=(m_1,\ldots,m_r) \in \P_{\pi/4}^-(n)$. Our aim here, is to associate to $m$ another composition $m'$ such that $P_{m'}\in \P_{\pi/4}^+(n)$ and \linebreak[4] $|CR_{\pi/4}(m)| \leq |CR_{\pi/4}(m')|$. \\
		
		First, let's prove that $\exists j\in \{1,\ldots,r \}$ such that $m_j \in E_3 \cup E_5$. Suppose the opposite, since $m \in \P_{\pi/4}^-(n)$ by Proposition \ref{Car}, the number $\# [E_1 \cap \{m_1,\ldots,m_r \}]$ is odd (there are no elements in $E_7 \cup E_8$ because $n \leq g < 7$), but the other numbers in $m$ are all even ($2$, $4$, $6$), then $m_1+\cdots+m_r = n$ is odd. This contradicts the supposition that $n$ is even. We conclude that $\exists j\in \{1,\ldots,r \}$ such that $m_j \in E_3 \cup E_5$.

		\begin{itemize}
			\item If $m$ contains $m_j \in E_5$, we associate to
			$$m=(m_1,\ldots,m_{j-1},m_j,m_{j+1},\ldots,m_r)$$
			$m'$ such that
			$$m'=(m_1,\ldots,m_{j-1},m_{j_1},m_{j_2},m_{j+1},\ldots,m_r)$$
			with $m_{j_1} \in E_1$ and $m_{j_2} \in E_4$.
			It is clear, by Proposition \ref{Car}, that \linebreak[4] $P_{m'}\in \P_{\pi/4}^+(n)$. Moreover, we have
			
			\begin{equation*}
				\begin{split}
					CR_{\pi/4}(m)=(-1)^r \cdot 2^r \cdot 2^{n/2}\prod_{s=1}^{j-1} \frac{(g-1)cos(m_s \cdot \pi/4)+cos(m_s \cdot \pi/2)}{m_1+\cdots+m_s} \cdot \frac{-(g-1) \cdot \sqrt{2}/2}{m_1+\cdots+m_j} \cdot \\ \prod_{s=j+1}^{r} \frac{(g-1)cos(m_s 
						\cdot \pi/4)+cos(m_s \cdot \pi/2)}{m_1+\cdots+m_s} 
				\end{split}
			\end{equation*}
			
			and
			
			\begin{equation*}
				\begin{split}
					CR_{\pi/4}(m')=(-1)^{r+1} \cdot 2^{r+1} \cdot 2^{n/2}\prod_{s=1}^{j-1} \frac{(g-1)cos(m_s \cdot \pi/4)+cos(m_s \cdot \pi/2)}{m_1+\cdots+m_s} \cdot \\ \frac{-(g-2)}{m_1+\cdots+m_{j-1}+ m_{j_1}} \cdot  \frac{(g-1) \cdot \sqrt{2}/2}{m_1+\cdots+m_j} \cdot \prod_{s=j+1}^{r} \frac{(g-1)cos(m_s \cdot \pi/4)+cos(m_s \cdot \pi/2)}{m_1+\cdots+m_s}
				\end{split}
			\end{equation*}
			since 
			\begin{equation}
				2 \cdot \frac{(g-2)}{m_1+\cdots+m_{j-1}+ m_{j_1}} \cdot \frac{(g-1) \cdot \sqrt{2}/2}{m_1+\cdots+m_j} > \frac{(g-1) \cdot \sqrt{2}/2}{m_1+\cdots+m_j} \label{In5}
			\end{equation}
			($m_1+\ldots+m_{j-1}+ m_{j_1} \leq n-4$ and $n \leq g$) we conclude that \linebreak[4] $|CR_{\pi/4}(m)|<|CR_{\pi/4}(m')|$.

			\item If $m$ contains $m_j \in E_3$, we associate to
			$$m=(m_1,\ldots,m_{j-1},m_j,m_{j+1},\ldots,m_r)$$
			$m'$ such that
			$$m'=(m_1,\ldots,m_{j-1},m_{j_1},m_{j_2},m_{j_3},m_{j+1},\ldots,m_r)$$
			with $m_{j_1}, m_{j_2}, m_{j_3} \in E_1$. It is clear, by Proposition \ref{Car}, that $P_{m'}\in \P_{\pi/4}^+(n)$. Moreover, we have

			\begin{equation*}
				\begin{split}
					CR_{\pi/4}(m)=(-1)^r \cdot 2^r \cdot 2^{n/2}\prod_{s=1}^{j-1} \frac{(g-1)cos(m_s \cdot \pi/4)+cos(m_s \cdot \pi/2)}{m_1+\cdots+m_s} \cdot \frac{-(g-1) \cdot \sqrt{2}/2}{m_1+\cdots+m_j} \cdot  \\ \prod_{s=j+1}^{r} \frac{(g-1)cos(m_s \cdot \pi/4)+cos(m_s \cdot \pi/2)}{m_1+\cdots+m_s} 
				\end{split}
			\end{equation*}
			
			and
			
			\begin{equation*}
				\begin{split}
					CR_{\pi/4}(m')=(-1)^{r+2} \cdot 2^{r+2} \cdot 2^{n/2}\prod_{s=1}^{j-1} \frac{(g-1)cos(m_s \cdot \pi/4)+cos(m_s \cdot \pi/2)}{m_1+\cdots+m_s} \cdot \\ \frac{(g-1) \cdot \sqrt{2}/2}{m_1+\cdots+m_{j-1}+ m_{j_1}} \cdot \frac{(g-1) \cdot \sqrt{2}/2}{m_1+\cdots+m_{j_2}} \cdot \frac{(g-1) \cdot \sqrt{2}/2}{m_1+\cdots+m_j} \cdot \\ \prod_{s=j+1}^{r} \frac{(g-1)cos(m_s \cdot \pi/4)+cos(m_s \cdot \pi/2)}{m_1+\cdots+m_s}
				\end{split}
			\end{equation*}
			
			since 
			\begin{equation}
				4 \cdot \frac{(g-1) \cdot \sqrt{2}/2}{m_1+\cdots+m_{j-1}+ m_{j_1}} \cdot \frac{(g-1) \cdot \sqrt{2}/2}{m_1+\cdots+m_{j_2}} \cdot \frac{(g-1) \cdot \sqrt{2}/2}{m_1+\cdots+m_j} > \frac{(g-1) \cdot \sqrt{2}/2}{m_1+\cdots+m_j} \label{In3}
			\end{equation}
			we conclude that $|CR_{\pi/4}(m)|<|CR_{\pi/4}(m')|$.
			
		\end{itemize}
		
		Notice that in both inequalities \eqref{In5} and \eqref{In3}, we can have the stronger conclusion: $2 \cdot |CR_{\pi/4}(m)|<|CR_{\pi/4}(m')|$ (because of the coefficients $2$ and $4$). That means that we can associate at least two negative compositions with one positive. \\ 
		That is why we have limited the theorem to $1 \leq g \leq 6$. There are not enough compositions to be required to associate more than two compositions with only one, except for $g=6$ where we associate 
		$$(1,1,1,3)\ \hbox{with}\ (1,1,1,1,1,1)$$
		$$(1,1,3,1)\ \hbox{with}\ (1,1,1,1,1,1)$$
		$$(1,3,1,1)\ \hbox{with}\ (1,1,4)$$
		$$(1,5)\ \hbox{with}\ (1,1,4)$$
		$$(3,1,1,1)\ \hbox{with}\ (1,4,1)$$
		$$(5,1)\ \hbox{with}\ (1,4,1)$$
		we prove this easily as we shown above. We conclude that $a_{n,\pi/4} > 0$ if $n$ is even.\\
		Similarly, we prove that $a_{n,\pi/4} < 0 $ if $n$ is odd, and the exception here is for $g=5$:
		$$(1,1,3)\ \hbox{with}\ (1,1,1,1,1)$$
		$$(1,3,1)\ \hbox{with}\ (1,1,1,1,1)$$
		$$(3,1,1)\ \hbox{with}\ (1,4)$$
		$$(5)\ \hbox{with}\ (1,4)$$
		
		Let's prove (c). Suppose $n$ is even, let $m=(m_1,\ldots,m_r) \in \P(n-1)$,
		\begin{itemize}
			\item If $m=(m_1,\ldots,m_r) \in \P_{\pi/4}^+(n-1)$, let $m'=(m'_1,\ldots,m'_{r'}) \in \P_{\pi/4}^-(n-1)$ such that they are associated as we saw previously. There is $M$ and $M'$ such that $M=(m_1,\ldots,m_r,1) \in \P_{\pi/4}^-(n)$ and $M'=(m'_1,\ldots,m'_{r'},1) \in \P_{\pi/4}^+(n)$ (if we add "$1$" to a composition, its sign changes by Proposition \ref{Car}). It is clear that $M$ and $M'$ are associated then
			$$ |CR_{\pi/4}(m)|<|CR_{\pi/4}(m')|\ \hbox{, }\ |CR_{\pi/4}(M)|<|CR_{\pi/4}(M')| $$
			and 
			$$ |CR_{\pi/4}(m)+CR_{\pi/4}(m')| < |CR_{\pi/4}(M)+CR_{\pi/4}(M')|$$
			Since $ 2(g-1)/n . |CR_{\pi/4}(m)+CR_{\pi/4}(m')| = |CR_{\pi/4}(M)+CR_{\pi/4}(M')|$ (see the first part of the proof).
			\item let $m'=(m'_1,\ldots,m'_{r'}) \in \P_{\pi/4}^-(n-1)$ that is not associated with any $\P_{\pi/4}^+(n-1)$ elements (it is possible by Proposition \ref{Ordre}), then $$M'=(m'_1,\ldots,m'_{r'},1)\  \hbox{or}\ M'=(1,m'_1,\ldots,m'_{r'})$$ are in $\P_{\pi/4}^+(n)$ and one of them is associated with no element of $\P_{\pi/4}^-(n)$ (because $g \leq 6$, we can avoid the paths $1,1,1$ and $1,4$ ). Moreover, $ |CR_{\pi/4}(m')|<|CR_{\pi/4}(M')|$.
		\end{itemize}
		We conclude that 
		$$ |a_{n-1,\pi/4}| = |\sum_{m \in \P(n-1)} CR_{\pi/4}(m)| < |a_{n,\pi/4}| = |\sum_{m \in \P(n)} CR_{\pi/4}(m)|.$$
		
	\end{proof}
	
	\begin{remark}
		\begin{itemize}
			\item In fact, Theorem \ref{SignAi} is true for $ 1 \leq g$ and the proof is still valid while we know how to deal with the exceptions, as we have seen for $g=5$ and $6$. The problem is that the number of the exceptions is quickly very large. We keep seeking for a stronger proof with no restriction for $g$, the idea is to organize $\P(n)$ following an ascending order of $CR_{\theta}(m)$ for analysing their distribution (Proposition \ref{Ordre} is a good beginning). 
			\item We know that if $N_1 \geq g+1$, there exists a non-special divisor of degree $g-1$. Moreover we know all the curves which contain this kind of divisors over $\F_2$ for $g=1$ and $2$ (see \cite{balb}). Since for defect 2 curves over $\F_2$ with $g \geq 3$ one has 
			$$|N_1-3|=2g-2\ \Rightarrow N_1 = 2g+1 \Rightarrow N_1 \geq g+1$$
			the existence of these divisors is obvious. Thus, we do not need to use the coefficients $a_n$ for this purpose. Nevertheless, for the defect 3 curves over $\F_2$ with $g \geq 3$, one has
			$$|N_1-3|=2g-3\ \Rightarrow N_1 = 2g\ \hbox{or}\ N_1=-2g+6 $$
			our next goal, is to use the method of section \ref{exam} to prove the existence of non-special divisor of degree $g-1$ for these curves (the case $N_1=-2g+6$ with $g=3$).  
		\end{itemize}
	\end{remark}
	
	\section*{Acknowlegements}
	I would like to express my deepest appreciation and gratitude to my PhD supervisors Stéphane BALLET and Julia PIELTANT  for their priceless advice, guidance and continues support. Indeed, their involvement highly contributed to improve the paper. I would like to extend my sincere thanks to my supportive wife, who had always believed in my professional projects and my parents who helped me through different stages of my life.
	
	\bibliography{BallKoutPiel_divNS_biblio.bib}

\end{document}